\documentclass{article}

\usepackage{arxiv}

\usepackage[utf8]{inputenc} % allow utf-8 input
\usepackage[T1]{fontenc}    % use 8-bit T1 fonts
\usepackage{hyperref}       % hyperlinks
\usepackage{url}            % simple URL typesetting
\usepackage{booktabs}       % professional-quality tables
\usepackage{amsfonts}       % blackboard math symbols
\usepackage{nicefrac}       % compact symbols for 1/2, etc.
\usepackage{microtype}      % microtypography
\usepackage{lipsum}
\usepackage{amsmath}

\usepackage{amsfonts,amssymb}
\usepackage{graphicx}
\usepackage{amsthm}
\usepackage{xcolor}

\newcommand{\R}{\mathbb{R}}

\newtheorem{theorem}{Theorem}[section]
\newtheorem{proposition}{Proposition}[section]
\newtheorem{lemma}{Lemma}[section]

\usepackage{cite}

\newcommand{\bb}{\begin{equation}}
\newcommand{\ee}{\end{equation}}
\newcommand{\ba}{\begin{array}}
\newcommand{\ea}{\end{array}}

\usepackage[all]{xy}

\newcommand{\sobolev}[2]{\Vert #1\Vert_{H^{#2}}}
\newcommand{\hm}[3]{\vert\vert\vert #1\vert\vert\vert_{E_{#2,#3}}}
\newcommand{\km}[3]{\Vert #1\Vert_{#2,#3}}
\newcommand{\kma}[2]{\Vert #1\Vert_{#2,2,m}}

\numberwithin{equation}{section}

\title{Global analytic solutions of a pseudospherical Novikov equation}

\author{
 Priscila~Leal~da Silva\\
  Departamento de Matemática, Universidade Federal de São Carlos, Brazil.\\
 Department of Mathematical Sciences, School of Science, Loughborough University, Loughborough, UK.\\
  \texttt{priscilals@ufscar.br and pri.leal.silva@gmail.com} \\
}

\begin{document}
\maketitle

\begin{abstract}
In this paper we consider a Novikov equation, recently shown to describe pseudospherical surfaces, to extend some recent results of regularity of its solutions. By making use of the global well-posedness in Sobolev spaces, for analytic initial data in Gevrey spaces we prove some new estimates for the solution in order to use the Kato-Masuda Theorem and obtain a lower bound for the radius of spatial analyticity. After that, we use embeddings between spaces to then conclude that the unique solution is, in fact, globally analytic in both variables. Finally, the global analyticity of the solution is used to prove that it endows the strip $(0,\infty)\times \R$ with a global analytic metric associated to pseudospherical surfaces obtained in \cite{SfF}.
\end{abstract}

% keywords can be removed
\keywords{Global well-posedness, analytic solutions, pseudospherical surfaces}

%\ams{35B60 \and 35B06 \and 35Q35}

\section{Introduction}

In this paper we consider the Cauchy problem
\begin{align}\label{nonlocal}
u_t =\partial_x u^2 +\partial_x\Lambda^{-2}\left(u^2 + \partial_x u^2\right),\quad u(0,x):=u_0(x),
\end{align}
where $u=u(t,x)$ and $\Lambda^{-2}$ denotes the inverse of the Helmholtz operator $\Lambda^2=1-\partial_x^2$. Equation \eqref{nonlocal} was originally deduced in \cite{Novikov} through the perturbative symmetry approach as one of the symmetry-integrable generalizations of the Camassa-Holm equation \cite{CH}. More than a decade later, in \cite{FT} the authors proved that solutions of \eqref{nonlocal} describe nontrivial families of pseudospherical surfaces with associated $1$-forms $\omega_i = f_{i1}dx + f_{i2}dt$, $i=1,2,3$, given by
\begin{align*}
    &f_{11}=u-u_{xx},\quad f_{12} = 2u(u-u_{xx})+\frac{4}{m_1}uu_x -2u_x^2 -2u^2,\\
    &f_{21} = \mu(u-u_{xx})\pm m_1\sqrt{1+\mu^2},\quad f_{22}=\mu\left[2u(u-u_{xx})+\frac{4}{m_1}uu_x - 2u_x^2-2u^2\right],\\
    &f_{31}=\pm \sqrt{1+\mu^2}(u-u_{xx})+m_1\mu,\quad f_{32}=\pm \sqrt{1+\mu^2}\left[2u(u-u_{xx})+\frac{4}{m_1}uu_x-2u_x^2-2u^2\right],
    \end{align*}
where $\mu$ is an arbitrary real parameter and $m_1=-2,1$. For $m_1=-2$, the corresponding conservation laws are nontrivial and the AKNS formulation \cite{AKNS}
$$\Psi_x(t,x) = X \Psi(t,x),\quad \Psi_t(t,x)=T\Psi(t,x),\quad X_t-T_x+[X,T]=0,$$
can be formulated for the $\texttt{sl}(2;\mathbb{R})$ matrices
$$X=\frac{1}{2}\begin{pmatrix}
f_{21}&f_{11}-f_{31}\\
f_{11}+f_{31} & -f_{21}
\end{pmatrix},\quad
T=\frac{1}{2}\begin{pmatrix}
f_{22}&f_{12}-f_{32}\\
f_{12}+f_{32} & -f_{22}
\end{pmatrix}.$$

In recent works, equation \eqref{nonlocal} was also studied from the point of view of qualitative properties of its solutions. Firstly we mention \cite{MLGL}, where the authors used Friedrichs mollifiers to prove that \eqref{nonlocal} is locally well-posed for initial data in $H^s(\R)$, with $s>3/2$, even though the data-to-solution map is not uniformly continuous. Equation \eqref{nonlocal} is also known to be well-posed in Besov spaces \cite{LY} leading to the global extension of the the local solution provided that $m_0(x):= u_0(x) - u_0''(x)$ is nonnegative, and also exhibitting conditions for wave-breaking. Of more importance to us and the main motivation for this paper, in \cite{MLGL} the authors showed that there exists a unique local analytic solution for the Cauchy problem with initial data in the Himonas-Misiolek space \cite{misiolek}. Moreover, some of these results combined with \cite{FT} lead to the explicit construction of the metric
\begin{align}\label{metric}
    \begin{aligned}
        g=&\left[(u-u_{xx})^2+(\mu(u-u_{xx})\pm m_1\sqrt{1+\mu^2})^2\right]dx^2\\
        &+2\left(2u(u-u_{xx})+\psi\right)\left[(1+\mu^2)(u-u_{xx})\pm m_1\mu\sqrt{1+\mu^2}\right]dxdt\\
        &+(1+\mu^2)\left(2u(u-u_{xx})+\psi\right)^2dt^2,
    \end{aligned}
\end{align}
where $\psi = 4(uu_x)/m_1-2u_x^2-2u^2$ and $\mu,m_1$ are given as before, for analytic pseudospherical surfaces, see \cite{SfF}.

Equation (1.1) has been referred to by various names, including generalised Camassa-Holm equation and generalised Degasperis-Procesi equation (see \cite{LY,MLGL}). However, due to the absence of the typical peakon solutions found in the Camassa-Holm and Degasperis-Procesi equations, as well as a known bi-Hamiltonian formulation, this equation appears to be significantly different from both, at least from the perspective of integrable systems. Consequently, we choose to refer to it as ``pseudospherical Novikov equation'', as discussed in references \cite{DF,FT}, where it was examined from a geometric standpoint.

The main result of this paper, as stated in Theorem \ref{thr1} below, generalizes the result in \cite{MLGL} by showing that the solution is, in fact, globally analytic in both variables. For that we restrict our analysis to analytic initial data in analytic Gevrey spaces to then extend the solution by making use of global well-posedness in Sobolev spaces. Given $u\in L^2(\R)$, we say that $u$ is in the Gevrey space $G^{\sigma,s}(\R)$, for $\sigma>0$ and $s\in\R$, if the norm
$$\Vert u\Vert_{G^{\sigma,s}}:=\Vert e^{\sigma|D_x|}u\Vert_{H^s} = \left(\int_{\R} e^{2\sigma |\xi|}(1+|\xi|^2)^s |\hat{u}(\xi)|^2d\xi\right)^{1/2}$$
is finite, where $\hat{u}$ denotes the spatial Fourier transform
$$\hat{u}(\xi) = \frac{1}{\sqrt{2\pi}}\int_{\R}e^{-ix\xi}u(x)dx.$$
Following the Paley-Wiener characterization \cite{Kat}, a function in $G^{\sigma,s}(\R)$ is the real line restriction a holomorphic function on a strip of width $2\sigma$. In the case where $\sigma=1$, functions are analytic and for this purpose we will consider our initial data in $G^{1,s}(\R)$ for $s$ greater than a certain positive number.

While the form \eqref{nonlocal} is very convenient to use algebra properties and obtain local results for well-posedness, it is not the best for our main purpose. Through the identity $\Lambda^{-2}\partial_x^2 = \Lambda^{-2}-1$, we can rewrite \eqref{nonlocal1} as
\begin{align}\label{nonlocal1}
u_t =\partial_x u^2 -u^2+\partial_x\Lambda^{-2}u^2 + \Lambda^{-2} u^2,
\end{align}
and this form will be used to prove the following main result:

\begin{theorem}\label{thr1}
    Given $u_0\in G^{1,s}(\R)$, with $s>3/2$, if $m_0(x)$ is non-negative, then the Cauchy problem of \eqref{nonlocal} has a unique global analytic solution $u\in C^{\omega}([0,\infty)\times \R)$. Moreover, the radius of spatial analyticity satisfies the lower bound
    \begin{align}\label{estRad}
        r(t) \geq L_3e^{-L_1e^{L_2t}},
    \end{align}
    where the positive parameters $L_1$ and $L_3$ depend only on the initial data and $L_2$ depend on the $H^2(\R)$-norm of the solution.
\end{theorem}

As it will be clear in the remaining of this paper, both the index $3/2$ and the non-negativeness of the initial momentum are due to global well-posedness in Sobolev spaces and not from the `optimal' choice of initial data (see Theorem \ref{thr2}) in the Gevrey space. Moreover, the persistence of the sign of the initial momentum $m_0(x)$ for global well-posedness of equations of this type is a remarkable characteristic. We stress out, however, that here it is needed the stronger condition $m_0(x)\geq 0$ to guarantee that, from the inequality
\begin{align}\label{diffeo}
m(t,q(t,x))\geq m_0(x) \exp\left(-2\int_0^t(u-3u_x)(s,q(t,x))ds\right),
\end{align}
where $q$ is a certain increasing diffeomorphism in the line (see \cite{LY,SfF}) uniquely defined for $t$ in the existence interval for the solution $u$, the solution $u$ is nonnegative, and global well-posedness will follow from arguments in Besov spaces, see \cite{LY} for more details.

%In the context of hydrodynamic applications, analyticity is a crucial ingredient to prove an intrinsic characterization of symmetric waves, see \cite{Escher} for more details. Here, we observe that our unique space-time analytic solutions provided by Theorem \ref{teo1.4} are not necessarily traveling waves, which then provides an interesting and general result.

%We initially observe that the bound for the radius of spacial analyticity presented in Theorem \ref{teo1.4} depends only on the initial value $u_0$ and 

In the proof of Theorem \ref{thr1} we will obtain the very rough estimate \eqref{estRad} for the radius of spatial analyticity $r(t)$ for every $T>0$ fixed and $t\in[0,T]$. The double exponential provokes a very strong and fast decay, meaning that the estimate is rather trivial. In this sense, even though Theorem \ref{thr1} guarantees global analyticity of solutions, the problem of obtaining a better estimate for the radius $r(t)$ is still open. With the exception of very few cases known in the literature \cite{CP,HP,HP1}, this improvement is a very challenging and difficult problem.

Theorem \ref{thr1} can be used to study the metric \eqref{metric}. In fact, we can prove the following result:

\begin{theorem}\label{thr1a}
    Given $u_0\in G^{1,s}(\R)$ such that the initial momentum $m_0(x)\in L^1(\R)$ is positive, then for any $\epsilon>0$ the analytic solution $u\in C^{\omega}([0,\infty)\times \R)$ endows the strip $\mathcal{S} = (0,\epsilon)\times \R$ with the analytic metric \eqref{metric}.
\end{theorem}

Theorem \ref{thr1a} is a generalization of Theorem 5.1 of \cite{SfF} and will use, along with Theorem \ref{thr1}, the positiveness of the solution $u$ and the conservation of the $L^1(\R)$-norm of $u$. This is why we are requiring $m_0$ to be in $L^1(\R)$, a condition that is missing in \cite{SfF} but is crucial to the proof.

The paper is structured as the following: in Section \ref{sec2} we will present more details about Gevrey and auxiliary spaces that will be needed to prove Theorem \ref{thr1} and in Section \ref{sec3} we will prove Theorem \ref{thr1} by making use of the global well-posedness of \eqref{nonlocal} in Sobolev spaces, the Kato-Masuda machinery and embeddings of the spaces discussed in Section \ref{sec2}. Finally, in Section \ref{sec4} we will prove Theorem \ref{thr1a}.

\section{Function spaces and  well-posedness}\label{sec2}

In this section we will briefly discuss the theory of function spaces required and enunciate the well-posedness results that will be used to prove Theorem \ref{thr1}.

The first important property is that the Gevrey spaces form a decreasing scale of Banach spaces, that is, if $0<\sigma'<\sigma$, then $\Vert\cdot \Vert_{G^{\sigma',s}}\leq \Vert\cdot \Vert_{G^{\sigma,s}}$ and $G^{\sigma,s}(\R)\hookrightarrow G^{\sigma',s}(\R)$. Moreover, $\Vert \varphi\Vert_{G^{\sigma,s}}\leq \Vert \varphi\Vert_{G^{\sigma,s'}}$ for any real choices $s'>s$, and because $\Vert \varphi\Vert_{G^{\sigma,s}}=\Vert e^{\sigma |D_x|}\varphi\Vert_{H^{s}}$, it is easy to see that it is an algebra for $\sigma>0$ and $s>1/2$ so that
\begin{align}\label{}
\Vert\varphi\psi\Vert_{G^{\sigma,s}}\leq c_s\Vert\varphi\Vert_{G^{\sigma,s}}\Vert\psi\Vert_{G^{\sigma,s}},
\end{align}
for some positive constant $c_s$ depending only on $s$, and the properties
\begin{align}\label{eqlemma2.2}
        &\Vert\partial_x\varphi\Vert_{G^{\sigma',s}} \leq \frac{e^{-1}}{\sigma-\sigma'}\Vert\varphi\Vert_{G^{\sigma,s}},\quad s\geq 0, \quad 0<\sigma'<\sigma\leq 1,\\
        &\Vert\partial_x\varphi\Vert_{G^{\sigma,s}} \leq \Vert\varphi\Vert_{G^{\sigma,s+1}},\quad s\in\mathbb{R},\\
        &\Vert\Lambda^{-2}\varphi\Vert_{G^{\sigma,s}}=\Vert\varphi\Vert_{G^{\sigma,s-2}},\quad s\in\mathbb{R},
    \end{align}
can be easily determined, see \cite{BHP}. In the inequalities above and throughout the paper, $\Lambda^{-2}$ is defined by $\Lambda^{-2}\varphi = g\ast \varphi(x)$, where $g(x)=e^{-|x|}/2$ and $\ast$ is the usual convolution, for all $\varphi\in H^{s}(\R)$.

Given an initial data $u_0(x)\in G^{1,s}(\R)$, for $F(u)$ given by the right-hand side of \eqref{nonlocal}, $s>1/2$ and $0<\sigma<1$, we can use the relations above to obtain
$$\Vert F(u_0)\Vert_{G^{\sigma,s}}\leq \frac{M}{1-\sigma},$$
and, if for a given $R>0$ we $\Vert u-u_0\Vert_{G^{1,s}}<R$ and $\Vert v-u_0\Vert_{G^{1,s}}<R$, also
$$\Vert F(u)-F(v)\Vert_{G^{\sigma,s}}\leq \frac{L}{1-\sigma}\Vert u-v\Vert_{G^{1,s}},$$
where $M=3c_se^{-1}\Vert u_0\Vert_{G^{1,s}}^2$ and $L=6c_se^{-1}(R +\Vert u_0\Vert_{G^{1,s}})$, and $c_s$ is the constant provided by the algebra property.

Based on the two estimates, we can use the following result,  see Theorem 1 of \cite{BHPlocal} and references therein, to obtain local well-posedness of \eqref{nonlocal}:

\begin{proposition}[\textsc{Autonomous Ovsyannikov Theorem}]\label{teo1.5}
Let $X_{\delta}$ be a scale of decreasing Banach spaces for $0<\delta \leq 1$, that is, $X_\delta \subset X_{\delta'}, \Vert \cdot \Vert_{\delta'}\leq \Vert \cdot \Vert_{\delta}, 0<\delta'<\delta\leq 1,$ and consider the Cauchy problem
\begin{align}\label{1.0.8}
    \begin{cases}
    \displaystyle{\frac{du}{dt} = G(u(t))},\\
    u(0) = u_0.
    \end{cases}
\end{align}
Given $\delta_0\in(0,1]$ and $u_0\in X_{\delta_0}$, assume that $G$ satisfies the following conditions:
\begin{enumerate}
    \item For $0<\delta'<\delta<\delta_0$, $R>0$ and $a>0$, if the function $t\mapsto u(t)$ is holomorphic on $\{t\in \mathbb{C}; 0<|t|<a(\delta_0-\delta)$ with values in $X_{\delta}$ and $\sup\limits_{t<a(\delta_0-\delta)}\Vert u-u_0\Vert_{\delta}<R$, then the function $t\mapsto G(t,u(t))$ is holomorphic on the same set with values in $X_{\delta'}$.
    
    \item $G:X_{\delta}\rightarrow X_{\delta'}$ is well defined for any $0<\delta'<\delta<\delta_0$ and for any $R>0$ and  $u,v\in B(u_0,R)\subset X_{\delta}$, there exist positive constants $L$ and $M$ depending only on $u_0$ and $R$ such that
    \begin{align*}
        &\Vert G(u) - G(v)\Vert_{\delta'}\leq \frac{L}{\delta-\delta'}\Vert u-v\Vert_{\delta},\quad \Vert G(u_0)\Vert_{\delta}\leq \frac{M}{\delta_0-\delta},
    \end{align*}
\end{enumerate}
$0<\delta<\delta_0$. Then for
\begin{align}\label{lifespanAOT}
    T = \frac{R}{16LR+8M}
\end{align}
the initial value problem \eqref{1.0.8} has a unique solution $u\in C^{\omega}([0,T(\delta_0-\delta)),X_{\delta})$, for every $\delta\in (0,\delta_0)$, satisfying
\begin{align}\label{sup}
    \sup\limits_{|t|<T(\delta_0-\delta)}\Vert u(t)-u_0\Vert_{\delta}<R,\quad 0<\delta<\delta_0.
\end{align}
\end{proposition}

From the Autonomous Ovsyannikov Theorem we conclude that for $s>1/2$, $\sigma\in(0,1),$ $R=\Vert u_0\Vert_{G^{1,s}}$ and $$T = \frac{e}{216 c_s\Vert u_0\Vert_{G^{1,s}}}$$ there exists a unique solution $u\in C^{\omega}([0,T(1-\sigma)), G^{\sigma,s}(\R))$, completing the proof of the following important result:

\begin{theorem}\label{thr2}
     Given $u_0\in G^{1,s}(\R)$, with $s>1/2$, there exists $T=T(\Vert u_0\Vert_{G^{1,s}})>0$ such that for every $\sigma\in (0,1)$ the initial value problem \eqref{nonlocal} has a unique solution $u\in C^{\omega}([0,T(1-\sigma)), G^{\sigma,s}(\R))$.
\end{theorem}

Consider also the Himonas-Misiolek space
$$E^{\sigma,m}(\R)=\{f\in C^{\infty}(\R); \hm{f}{\sigma}{m} = \sup\limits_{j\in\mathbb{Z}_+}\frac{\sigma^j(j+1)^2}{j!}\Vert \partial_x^jf\Vert_{H^{2m}}<\infty\}$$
for $\sigma>0$ and $m\in\mathbb{Z}_+$. For $0<\sigma<\sigma_0\leq 1$ in \cite{MLGL} the authors showed that for any given initial data $u_0\in E^{\sigma_0,1}(\R)$ the corresponding unique solution is locally analytic in time and belongs to $E^{\sigma,1}(\R)$ as a spatial function. This result can be extended for $m\geq 1$ in the Himonas-Misiolek space. We omit the proof for convenience, but it is essentially a repetition of the the arguments used for the proof of \eqref{thr2}, see also \cite{ProcEdin,misiolek,MLGL}, with properties
    \begin{align*}
        &\hm{\partial_x\varphi}{\sigma'}{m} \leq \frac{1}{\sigma-\sigma'}\hm{\varphi}{\sigma}{m},\\
        &\hm{\partial_x\varphi}{\sigma}{m} \leq \hm{\varphi}{\sigma}{m+1},\\
        &\hm{(1-\partial_x^2)^{-1}\varphi}{\sigma}{m+2}=\hm{\varphi}{\sigma}{m}.
    \end{align*}
for $0<\sigma'<\sigma\leq 1$, $m\geq1$ and $\varphi\in E_{\sigma,m}(\R)$, see \cite{ProcEdin}, to once again apply the Autonomous Ovsyanikov Theorem.

\begin{lemma}\label{lemma3}
    Given $u_0\in E^{\sigma_0,m}(\R)$, with $\sigma_0\in (0,1]$ and $m\geq 1$, then there exists $\epsilon=\epsilon(\hm{u_0}{\sigma_0}{m})>0$ and a unique solution $u\in C^{\omega}([0,\epsilon],E^{\sigma,m}(\R))$  of \eqref{nonlocal} for any $\sigma \in (0,\sigma_0)$.
\end{lemma}

It is worth mentioning that if we allow $t<0$, then by a reflection of $t$ and $x$ the interval in Lemma \ref{lemma3} can be extended to $[-\epsilon,\epsilon]$. This means that if we have an initial condition $u(T^{\ast},x)$, then we guarantee the existence of a unique solution defined for $t\in [T^{\ast}-\epsilon,T^{\ast}+\epsilon]$.

Another very important result is the global well-posedness of \eqref{nonlocal} in Sobolev spaces:

\begin{lemma}[Theorem 5.7 of \cite{LY}]\label{lemma4}
    Let $u_0\in H^s(\R)$, $s>3/2$. If $m_0\geq 0$, then the corresponding solution $u$ belongs to $C([0,\infty), H^s(\R))$.
\end{lemma}

Finally, we introduce the embeddings needed to extend local analyticity of the solution to global. As in \cite{KM}, define the Kato-Masuda spaces $A(r)$, for a fixed $r>0$, of functions that can be analytically extended to a function on a strip of width $r$ and endowed with the norm
\begin{align}\label{2.0.1a}
    \Vert f\Vert_{\sigma,s}^2 = \sum\limits_{j=0}^{\infty} \frac{1}{(j!)^{2}}e^{2\sigma j} \Vert\partial_x^jf\Vert_{H^s}^2,
\end{align}
for $s\geq 0$ and every $\sigma\in\R$ such that $e^{\sigma}<r$. Moreover, for $H^{\infty}(\R) := \bigcap\limits_{s\geq 0}H^{s}(\R)$, the embeddings $G^{\sigma,s}(\R)\hookrightarrow A(\sigma)\hookrightarrow H^{\infty}(\R)$ hold, see Lemma 2.3 and Lemma 2.5 in \cite{BHP} and Lemma 2.2 in \cite{KM}, for $\sigma>0$ and $s\geq 0$. Due to the embeddings, restriction of the initial data to analytic in $G^{1,s}(\mathbb{R})$ allows us to formulate the following result:

\begin{proposition}\label{prop5}
    Let $u_0\in G^{1,s}(\R)$, with $s>3/2$ and suppose that $m_0$ is non-negative. Then \eqref{nonlocal} has a unique solution $u\in C([0,\infty),H^{\infty}(\R))$.
\end{proposition}
\begin{proof}\,
    Given $u_0\in G^{\sigma,s}(\R)\hookrightarrow H^{\infty}(\R)$, in particular we have that $u_0\in H^{s'}(\R)$ for $s'\geq 2$ and Lemma \ref{lemma4} guarantees that the unique global solution is in $C([0,\infty),H^{s'}{\R)})$ for $s'\geq 2$. Moreover, since $H^{2}(\R)\subset H^{s'}(\R)$ for $s'\in[0,2]$, we conclude that $u \in C([0,\infty),H^{s'}(\R))$ for $s'\in[0,2]$, which shows that $u \in C([0,\infty),H^{\infty}(\R))$.
\end{proof}

\section{Global analytic well-posedness}\label{sec3}

In this section we prove Theorem \ref{thr1}. The first step is to show that the global solution of Proposition \ref{prop5} is globally analytic in space and provide a rough bound for the radius of spatial analyticity through Kato-Masuda Theorem. For that we will consider the more convenient norm 
$$\kma{u}{\sigma}^2 = \sum\limits_{j=0}^m\frac{1}{(j!)^{2}}e^{2\sigma j} \sobolev{\partial_x^ju}{2}^2,\quad m\geq 0,$$
in $A(r)$ and recover \eqref{2.0.1a} as $m\to \infty$. Moreover, we observe that $\kma{u}{\sigma}\leq \km{u}{\sigma}{2}$.

Regarding our Cauchy problem \eqref{nonlocal}, given $u_0\in G^{1,s}(\R)$, with $s>3/2$, such that $m_0(x)$ is nonnegative, let $u$ be the corresponding global solution obtained in Proposition \ref{prop5}. Because $G^{1,s}(\R)\hookrightarrow A(1)$, let $\sigma_0 <0=:\bar{\sigma}$ so that $e^{\sigma_0}<1$. Moreover, for a fixed $T>0$, let $\mu = 1+\max\{\Vert u\Vert_{H^2}; t\in[0,T]\}$ and define $O = \{v\in H^{m+5}(\R); \Vert v\Vert_{H^{m+5}}<\mu\}\subset Z = H^{m+5}(\R)$ for $m\geq 0$. For $F:O\rightarrow X$ given by the right-hand side of \eqref{nonlocal1}, where $X= H^{m+2}(\R)$, and $$\Phi_{\sigma,m}(u) = \frac{1}{2}\kma{u}{\sigma}^2,$$ with $u\in H^{m+5}(\R)$, we have the right settings to use the Kato-Masuda Theorem, see Theorem 1 in \cite{KM} or Theorem 4.1 in \cite{BHP1}:

\begin{lemma}[\textsc{Kato-Masuda}]\label{prop4.2}
    Let $\{\Phi_{\sigma}:-\infty<\sigma<\bar{\sigma}\}$ be a family of real functions defined on an open set $O\subset Z$ for some $\bar{\sigma}\in\R$. Suppose that $F: O \rightarrow X$ is continuous and
    \begin{enumerate}
        \item[(a)] $D\Phi_{\cdot}(\cdot):\R\times Z\rightarrow \mathcal{L}(\R\times X; \R)$ given by
        $$D\Phi_{\sigma}(v)F(v):=\langle F(v)\,,\, D \Phi_{\sigma}(v)\rangle$$
        is continuous, where $D$ denotes the Fréchet derivative;
        \item[(b)] there exists $\bar{r}>0$ such that $$D\Phi_{\sigma}(v)F(v) \leq \beta(\Phi_{\sigma}(v)) + \alpha(\Phi_{\sigma}(v))\partial_{\sigma}\Phi_{\sigma}(v),$$
        for all $v\in O$ and some nonnegative continuous real functions $\alpha(r)$ and $\beta(r)$ well-defined for $-\infty<r<\bar{r}$.
    \end{enumerate}
    For $T>0$, let $u\in C([0,T];O)\cap C^1([0,T];X)$ be a solution of the initial value problem \eqref{nonlocal} such that there exists $b<\bar{\sigma}$ with $\Phi_{b}(u_0)<\bar{r}$. Finally, let $\rho(v)$ be the unique solution of
    \begin{align*}
        \begin{cases}
        \displaystyle{\frac{d\rho(t)}{dt} = \beta(\rho)},\\
        \rho(0) = \Phi_b(u_0),& t\geq 0.
        \end{cases}
    \end{align*}
    Then for
    $$\sigma(t) = b - \int_{0}^t \alpha(\rho(\tau))d\tau,\quad t\in[0,T_1],$$
    where $T_1>0$ is the lifespan of $\rho$, we have
    \begin{align}\label{5.0.2}
        \Phi_{\sigma(t)}(u) \leq \rho(t),\quad t\in[0,T'],\quad T'=\min\{T,T_1\}.
    \end{align}
\end{lemma}

For the settings above, it is easy to check that $F:O\rightarrow X$ defined as the right-hand side of \eqref{nonlocal1} is continuous and item $(a)$ of Kato-Masuda Theorem is satisfied. We shall now obtain the bound for item (b):

\begin{proposition}\label{prop7}
    Given $u\in H^{m+5}(\R), m\geq 0,$ for $\sigma\in \R$ we have the bound
    \begin{align*}
        \vert D\Phi_{\sigma,m}F(u)\vert \leq \bar{K}(\sobolev{u}{2})\Phi_{\sigma,m}(u) + \bar{\alpha}(\sobolev{u}{2},\Phi_{\sigma,m}(u))\partial_{\sigma}\Phi_{\sigma,m}(u),
    \end{align*}
    where $\bar{K}(p) = 144p$ and $\bar{\alpha}(p,q) =64q^{1/2}(4+3p).$
\end{proposition}
\begin{proof}\,
    We start by noticing that $\displaystyle{\frac{1}{2}\langle D\sobolev{\partial_x^ju}{2}^2\,,\,w\rangle = \langle \partial_x^jw\,,\, \partial_x^j u\rangle_{H^2}},$
    see \cite{BHP,KM}. Then by making $w = F(u)$ we obtain
    \begin{align}\label{innerproduct}
        \vert D\Phi_{\sigma,m}(u)F(u)\vert =&\left\vert \sum\limits_{j=0}^{m}\frac{e^{2\sigma j}}{(j!)^2}\langle\partial_x^j u\,,\,\partial_x^jF(u)\rangle_{H^2}\right\vert\nonumber\\
        \leq& 2\left\vert\sum\limits_{j=0}^{m}\frac{e^{2\sigma j}}{(j!)^2}\langle\partial_x^j u\,,\,\partial_x^j(uu_x)\rangle_{H^2} \right\vert+ \left\vert\sum\limits_{j=0}^{m}\frac{e^{2\sigma j}}{(j!)^2}\langle\partial_x^j u\,,\,\partial_x^{j+1}\Lambda^{-2}u^2\rangle_{H^2}\right\vert\nonumber\\
        & + \left\vert\sum\limits_{j=0}^{m}\frac{e^{2\sigma j}}{(j!)^2}\langle\partial_x^j u\,,\,\partial_x^ju^2\rangle_{H^2} \right\vert+ \left\vert\sum\limits_{j=0}^{m}\frac{e^{2\sigma j}}{(j!)^2}\langle\partial_x^j u\,,\,\partial_x^{j}\Lambda^{-2}u^2\rangle_{H^2}\right\vert.
    \end{align}
    From \cite{BHP1}, see equations (6.14) and (6.15) with $k=1$ in the aforementioned paper, we know that
    \begin{align}\label{eq3.1}
    \begin{aligned}
    \left\vert\sum\limits_{j=0}^{m}\frac{e^{2\sigma j}}{(j!)^2}\langle\partial_x^j u\,,\,\partial_x^j(uu_x)\rangle_{H^2} \right\vert \leq& \bar{K}_1(\sobolev{u}{2})\Phi_{\sigma,m}(u) \\
    &+ \alpha_1(\sobolev{u}{2},\Phi_{\sigma,m}(u))\partial_{\sigma}\Phi_{\sigma,m}(u),
    \end{aligned}
    \end{align}
    where $\bar{K_1}(p) = 32p$ and $\alpha_1(p,q) = 64(1+p)q^{1/2}$, and
    \begin{align}\label{eq3.2}
    \begin{aligned}
    \left\vert\sum\limits_{j=0}^{m}\frac{e^{2\sigma j}}{(j!)^2}\langle\partial_x^j u\,,\,\partial_x^{j+1}\Lambda^{-2}u^2\rangle_{H^2}\right\vert \leq& \frac{\bar{K}_1(\sobolev{u}{2})}{2}\Phi_{\sigma,m}(u) \\
    &+ \alpha_1(\sobolev{u}{2},\Phi_{\sigma,m}(u))\partial_{\sigma}\Phi_{\sigma,m}(u),
    \end{aligned}
    \end{align}
    so it remains to find similar bounds for the last two terms. We start with the first term:
    \begin{align}\label{eq3.3}
        \left\vert\sum\limits_{j=0}^{m}\frac{e^{2\sigma j}}{(j!)^2}\langle\partial_x^j u\,,\,\partial_x^ju^2\rangle_{H^2} \right\vert \leq& \vert\langle u\,,\,u^2\rangle_{H^2} \vert+\sum\limits_{j=1}^{m}\frac{e^{2\sigma j}}{(j!)^2}\left\vert\langle\partial_x^j u\,,\,\partial_x^ju^2\rangle_{H^2} \right\vert\nonumber\\
        \leq& 8\Vert u\Vert_{H^2}\Vert u\Vert^2_{H^2}+\sum\limits_{j=1}^{m}\frac{e^{2\sigma j}}{(j!)^2}\left\vert\langle\partial_x^j u\,,\,\partial_x^ju^2\rangle_{H^2} \right\vert\nonumber\\
        \leq&16\Vert u\Vert_{H^2}\Phi_{\sigma,m}(u)+\sum\limits_{j=1}^{m}\frac{e^{2\sigma j}}{(j!)^2}\left\vert\langle\partial_x^j u\,,\,\partial_x^ju^2\rangle_{H^2} \right\vert,
    \end{align}
    where in the last two inequalities we used the Cauchy-Schwarz inequality, the algebra property  $\Vert \varphi\psi\Vert_{H^2} \leq 8\Vert \varphi\Vert_{H^2}\Vert \psi\Vert_{H^2}$, see \cite{BHP1}, and the definition of $\Vert f\Vert_{\sigma,2,m}$ to conclude that $\Vert f\Vert_{H^2}\leq \Vert f\Vert_{\sigma,2,m}$. For the sum starting at $j=1$, after using the Leibniz formula for $\partial_x^j$ we can split the inner product of $H^2(\R)$ as  
    \begin{align*}
    \vert \langle \partial_x^j u, \partial_x^ju^2\rangle_{H^2}\vert =& \left\vert\left\langle \partial_x^ju \,,\, \sum\limits_{\ell=0}^j\binom{j}{\ell}\partial_x^{\ell}u\partial_x^{j-\ell}u \right\rangle_{H^2}\right\vert\\  
    \leq& \left\vert\langle \partial_x^j u\,,\, u\partial_x^ju\rangle_{H^2}\right\vert + \left\vert\left\langle \partial_x^ju \,,\, \sum\limits_{\ell=1}^j\binom{j}{\ell}\partial_x^{\ell}u\partial_x^{j-\ell}u \right\rangle_{H^2}\right\vert\\
    =:& (I) + (II),
    \end{align*}
    so that
    \begin{align}\label{eq3.4}
        \sum\limits_{j=1}^{m}\frac{e^{2\sigma j}}{(j!)^2}\left\vert\langle\partial_x^j u\,,\,\partial_x^ju^2\rangle_{H^2} \right\vert \leq \sum\limits_{j=1}^{m}\frac{e^{2\sigma j}}{(j!)^2}(I) + \sum\limits_{j=1}^{m}\frac{e^{2\sigma j}}{(j!)^2}(II).
    \end{align}
    The estimate for one $(I)$ is obtained through Cauchy-Schwarz and the algebra property:
    \begin{align*}
        (I) =& \vert\langle \partial_x^j u\,,\, u\partial_x^j u\rangle_{H^2}\vert        \leq\Vert \partial_x^j u\Vert_{H^2}\Vert u\partial_x^j u\Vert_{H^2}\leq 8 \Vert \partial_x^j u \Vert_{H^2}^2 \Vert u \Vert_{H^2},
    \end{align*}
    which means that
    \begin{align}\label{eq3.5}
    \sum\limits_{j=1}^{m}\frac{e^{2\sigma j}}{(j!)^2}(I) \leq& 8 \sum\limits_{j=0}^{m}\frac{e^{2\sigma j}}{(j!)^2}\Vert \partial_x^j u \Vert_{H^2}^2 \Vert u \Vert_{H^2} = 8 \Vert u\Vert_{H^2}\Vert u\Vert_{\sigma,2,m}^2 
    = 16\Vert u\Vert_{H^2} \Phi_{\sigma,m}(u).
    \end{align}
    
    The estimate for $(II)$ requires a little more work. In fact, we will need the following simpler version of Lemma 3.1 of \cite{KM}:

    \begin{lemma}[Lemma 3.1 of \cite{KM}]\label{lemma7}
        For $b_j=(j!)^{-1}e^{\sigma j}\Vert \partial_x^j u \Vert_{H^2}$, we have
        $$\sum\limits_{j=1}^m\sum\limits_{\ell=1}^j b_jb_{\ell}b_{j-\ell}\leq \Vert u\Vert_{\sigma,2,m}\partial_{\sigma}\Vert u\Vert_{\sigma,2,m}^2.$$
    \end{lemma}

    With Lemma \ref{lemma7} in hands, we estimate $(II)$ in the following way:
    \begin{align}\label{ast}
    \begin{aligned}
        (II)\leq& \sum\limits_{\ell=1}^j \frac{j!}{\ell!(j-\ell)!}\vert\left\langle \partial_x^ju \,,\,\partial_x^{\ell}u\partial_x^{j-\ell}u \right\rangle_{H^2}\vert \leq \sum\limits_{\ell=1}^j \frac{j!}{\ell!(j-\ell)!}\Vert \partial_x^ju\Vert_{H^2}\Vert \partial_x^{\ell}u\partial_x^{j-\ell}u\Vert_{H^2}\\
        \leq& 8 \sum\limits_{\ell=1}^j \frac{j!}{\ell!(j-\ell)!}\Vert \partial_x^ju\Vert_{H^2}\Vert \partial_x^{\ell}u\Vert_{H^2}\Vert\partial_x^{j-\ell}u\Vert_{H^2},
    \end{aligned}
    \end{align}
    where in the last inequality we once again made use of the algebra property in $H^2(\R)$. After rearranging terms, Lemma \ref{lemma7} yields
    \begin{align}\label{eq3.6}
    \begin{aligned}
        \left\vert\sum\limits_{j=1}^{m}\frac{e^{2\sigma j}}{(j!)^2}(II)\right\vert \leq& 8\sum\limits_{j=1}^{m}\sum\limits_{\ell=1}^j\frac{e^{2\sigma j}}{j!\ell!(j-\ell)!}\Vert \partial_x^ju\Vert_{H^2}\Vert \partial_x^{\ell}u\Vert_{H^2}\Vert\partial_x^{j-\ell}u\Vert_{H^2}\\
        =&8\sum\limits_{j=1}^{m}\sum\limits_{\ell=1}^j\left(\frac{e^{\sigma j}}{j!}\Vert \partial_x^ju\Vert_{H^2}\right)\left(\frac{e^{\sigma \ell}}{\ell!}\Vert \partial_x^{\ell}u\Vert_{H^2}\right)\left(\frac{e^{\sigma(j-\ell)}}{(j-\ell)!}\Vert\partial_x^{j-\ell}u\Vert_{H^2}\right)\\
        =& 8\sum\limits_{j=1}^{m}\sum\limits_{\ell=1}^j b_jb_{\ell}b_{j-\ell}\leq 8\Vert u\Vert_{\sigma,2,m}\partial_{\sigma}\Vert u\Vert_{\sigma,2,m}^2 \\
        \leq& 32 (\Phi_{\sigma,m}(u))^{1/2}\partial_{\sigma} \Phi_{\sigma,m}(u).
    \end{aligned}
    \end{align}
    
    As we plug \eqref{eq3.5} and \eqref{eq3.6} into \eqref{eq3.4} and then the respective result into \eqref{eq3.3}, we conclude that
    \begin{align}\label{eq3.7}
        \left\vert\sum\limits_{j=0}^{m}\frac{e^{2\sigma j}}{(j!)^2}\langle\partial_x^j u\,,\,\partial_x^ju^2\rangle_{H^2} \right\vert \leq \bar{K}_1(\Vert u\Vert_{H^2})\Phi_{\sigma,m}(u) + \alpha_2(\Phi_{\sigma,m}(u))\partial_{\sigma}\Phi_{\sigma,m}(u),
    \end{align}
    with $\alpha_2(q)=32q^{1/2}$.

    It remains to estimate the last term of \eqref{innerproduct}. Similarly to what was done previously, we have
    \begin{align*}
        \left\vert\sum\limits_{j=0}^{m}\frac{e^{2\sigma j}}{(j!)^2}\langle\partial_x^j u\,,\,\partial_x^{j}\Lambda^{-2}u^2\rangle_{H^2}\right\vert \leq& |\langle u\,,\,\Lambda^{-2}u^2\rangle_{H^2}| + \sum\limits_{j=1}^{m}\frac{e^{2\sigma j}}{(j!)^2}\left\vert\langle\partial_x^j u\,,\,\partial_x^{j}\Lambda^{-2}u^2\rangle_{H^2}\right\vert\\
        \leq& 16\Vert u\Vert_{H^2}\Phi_{\sigma,m}(u) +\sum\limits_{j=1}^{m}\frac{e^{2\sigma j}}{(j!)^2}(\tilde{I}) + \sum\limits_{j=1}^{m}\frac{e^{2\sigma j}}{(j!)^2}(\tilde{II}),
    \end{align*}
    where 
    $$(\tilde{I})=\left\vert\langle \partial_x^j u\,,\, \Lambda^{-2}(u\partial_x^ju)\rangle_{H^2}\right\vert,\quad (\tilde{II})= \left\vert\left\langle \partial_x^ju \,,\, \sum\limits_{\ell=1}^j\binom{j}{\ell}\Lambda^{-2}(\partial_x^{\ell}u\partial_x^{j-\ell}u) \right\rangle_{H^2}\right\vert.$$
    The term $\tilde{I}$ is estimated as
       \begin{align*}
        (\tilde{I}) = \left\vert\langle \partial_x^j u\,,\, \Lambda^{-2}(u\partial_x^ju)\rangle_{H^2}\right\vert \leq 8\Vert u\Vert_{H^2}\Vert \partial_x^ju\Vert_{H^2}^2,
    \end{align*}
    while we reobtain the previous case as we observe that
    \begin{align*}
        (\tilde{II}) \leq& \sum\limits_{\ell=1}^j \frac{j!}{\ell!(j-\ell)!}\vert\left\langle \partial_x^ju \,,\,\Lambda^{-2}(\partial_x^{\ell}u\partial_x^{j-\ell}u) \right\rangle_{H^2}\vert\\
        \leq& \sum\limits_{\ell=1}^j \frac{j!}{\ell!(j-\ell)!}\Vert\partial_x^j u\Vert_{H^2}\Vert\partial_x^{\ell}u\partial_x^{j-\ell}u\Vert_{H^0}\\
        \leq& \sum\limits_{\ell=1}^j \frac{j!}{\ell!(j-\ell)!}\Vert\partial_x^j u\Vert_{H^2}\Vert\partial_x^{\ell}u\partial_x^{j-\ell}u\Vert_{H^2}
    \end{align*}
    and use the corresponding estimates in \eqref{ast} and \eqref{eq3.6} for $(II)$ to conclude that 
    $$\sum\limits_{j=1}^{m}\frac{e^{2\sigma j}}{(j!)^2}(\tilde{II}) \leq 32 (\Phi_{\sigma,m}(u))^{1/2}\partial_{\sigma} \Phi_{\sigma,m}(u).$$
    Therefore, we obtain
    \begin{align}\label{eq3.9}
        \left\vert\sum\limits_{j=0}^{m}\frac{e^{2\sigma j}}{(j!)^2}\langle\partial_x^j u\,,\,\partial_x^{j}\Lambda^{-2}u^2\rangle_{H^2}\right\vert \leq& \bar{K}_1(\Vert u\Vert_{H^2})\Phi_{\sigma,m}(u) + \alpha_2(\Phi_{\sigma,m}(u))\partial_{\sigma}\Phi_{\sigma,m}(u)
    \end{align}
    for the same $\bar{K}_1(p)=32p$ and $\alpha_2(q)=32q^{1/2}$.
    
    Under substitution of the respective terms \eqref{eq3.1}-\eqref{eq3.3}, \eqref{eq3.7}-\eqref{eq3.9} in \eqref{innerproduct}, we conclude the result.
\end{proof}

We shall now complete the details of item $(b)$ of Kato-Masuda Theorem in a quite standard way. Let
\begin{align*}
    %&\mu = 1+\max\{\sobolev{u}{2}; t\in [0,T]\},\\
    &K = \bar{K}(\mu),\quad \beta(p) = Kp,\quad p\geq 0,\\
    &\rho(t) = \frac{1}{2}\km{u_0}{\sigma_0}{2}^2e^{Kt},\quad \rho_m(t) = \frac{1}{2}\kma{u_0}{\sigma_0}^2e^{Kt},\\
    &\bar{r} = 1+ \max\{\rho(t);t\in[0,T]\}, \quad \alpha(p) = \bar{\alpha}(\mu,p),
\end{align*}
where $\bar{K}$ and $\bar{\alpha}$ are given as in Proposition \ref{prop7}. Observe that $\alpha(p)$ and $\beta(p)$ are continuous functions for $p<\bar{r}$, $\rho_m(t)\leq \rho(t)$ for all $t\in[0,T]$, $\rho_m(t)\to \rho(t)$ uniformly and
\begin{align*}
    \bar{K}(\Vert u\Vert_{H^2}) \leq \bar{K}(\mu) = K, \quad\bar{\alpha}(\Vert v\Vert_{H^2},\Phi_{\sigma,m}(u))\leq \bar{\alpha}(\mu,\Phi_{\sigma,m}(u)) = \alpha(\Phi_{\sigma,m}(v))
\end{align*}
for $v\in O$. From Proposition \ref{prop7}, $v\in O$ then implies 
\begin{align*}
        \vert D\Phi_{\sigma,m}F(v)\vert \leq \beta(\Phi_{\sigma,m}(v)) + \alpha(\Phi_{\sigma,m}(u))\partial_{\sigma}\Phi_{\sigma,m}(u),
    \end{align*}
    and item $(b)$ is finally satisfied.

    For $b:=\sigma_0<\bar{\sigma}=0$, we have
    $$\Phi_{\sigma_0,m}(u_0) = \frac{1}{2}\Vert u_0\Vert^2_{\sigma_0,2,m}\leq \frac{1}{2}\Vert u_0\Vert^2_{\sigma_0,2} = \rho(0)\leq \bar{r}-1 < \bar{r}$$
    and, furthermore, $\rho_m(t)$ is a solution of the Cauchy problem
    \begin{align*}
    \frac{d}{dt} \rho_m(t) = \beta(\rho_m(t)),\quad \rho_m(0) = \Phi_{\sigma_0,m}(u_0),\quad t\geq 0.
    \end{align*}
    Kato-Masuda Theorem then states that for
    $$\sigma_m(t) = \sigma_0 - \int_0^t \alpha(\rho_m(\tau))d\tau,\quad t\in[0,T],$$
    we obtain $\Phi_{\sigma_m(t),m}(u) \leq \rho_m(t)
    \leq \rho(t) = \frac{1}{2}\Vert u_0\Vert^2_{\sigma_0,2}e^{Kt}$ for $t\in[0,T]$. By letting $m\to \infty$, we obtain $\Vert u\Vert_{\sigma(t),2}^2\leq \Vert u_0\Vert^2_{\sigma_0,2}e^{Kt}$ and $u(t)\in A(r_1)$ for $r_1 = e^{\sigma(t)}$, for $t\in[0,T]$. The radius of spatial analyticity $\sigma(t)$ is given by
\begin{align}\label{5.2.2}
    \sigma(t) = \sigma_0 - \int_0^t\alpha(\rho(\tau))d\tau = \sigma_0 - A(e^{Bt}-1),
\end{align}
where $A=\frac{4\sqrt{2}}{9\mu}(4+3\mu))\km{u_0}{\sigma_0}{2}$ and  $B=72\mu$. Since $\mu\geq 1,$ we can estimate $A\leq \frac{28\sqrt{2}}{9}\km{u_0}{\sigma_0}{2}=:L_1$ and, after setting $L_2:=B$, we obtain the following lower bound for the radius of spatial analyticity:
\begin{align}\label{eq3.12}
r(t) = e^{\sigma(t)} \geq L_3e^{-L_1e^{L_2t}}, \quad L_3:=e^{\sigma_0}e^{L_1}.    
\end{align}

In summary, the use of the Kato-Masuda Theorem showed that given $u_0\in G^{1,s}(\R)$ with $s>3/2$ such that $m_0(x)$ is non-negative, the global solution $u$ belongs to $C([0,\infty), A(r))$. Moreover, for each $T>0$ fixed, we have the lower bound \eqref{eq3.12} for the radius $r(t), t \in [0,T]$.

The remaining of the proof is standard \cite{BHP1,ProcEdin}.
\begin{enumerate}
    \item We now take a step back: for the same initial value $u_0$, let $\tilde{u}\in C^{\omega}([0,\tilde{T}(1-\sigma)), G^{\sigma,s}(\R))$ be the unique local Gevrey solution obtained by Theorem \ref{thr2} for $\sigma \in (0,1)$, $s>3/2>1/2$ and a certain $\tilde{T}>0$. By letting $$T=\frac{\tilde{T}(1-\sigma)}{2},$$ that is, $\sigma(T)=\sigma = 1-2T/\tilde{T}$, the embeddings $G^{\sigma(T),s}(\R)\hookrightarrow A(\sigma(T))\hookrightarrow H^{\infty}(\R)$ tell us that $\tilde{u}\in C^{\omega}([0,T],A(\sigma(T)))\subset C^{\omega}([0,T],H^{\infty}(\R))$ and then Proposition \ref{prop5} tells us that $\tilde{u}=u$ for $t\in[0,T]$ and $u\in C^{\omega}([0,T],A(\sigma(T)))$.

    \item Let $$T^{\ast} = \sup\{T>0, u\in C^{\omega}([0,T], A(\sigma(T))),\,\,\text{for some}\,\, \sigma(T)>0\}$$
    and assume that $T^{\ast}$ is finite. Because $C^{\omega}([0,T],A(\sigma(T)))\subset C^{\omega}([0,T],H^{\infty}(\R))$, it is due to Kato-Masuda that the initial data $u(T^{\ast})\in A(r)$ is well-defined. Let $\sigma_0<\min \{1, r_1/e\}$ and then the converse of Lemma 5.1 of \cite{BHP1} tells that $u(T^{\ast})\in A(r_1)\subset E_{\sigma_0,m}(\R),$ for $m\geq3$ and $\sigma_0\in(0,1].$ From Lemma \ref{lemma3}, there exist $\epsilon>0$ and a unique solution $\tilde{u}\in C^{\omega}([0,\epsilon];E_{\delta,m}(\R))$ for $0<\delta<\sigma_0$ such that $\tilde{u}(0) = u(T^{\ast})$. Conversely, the embedding $E_{\delta,m}(\R)\hookrightarrow A(\delta)\subset H^{\infty}(\R),$ yields
    $$\tilde{u}\in C^{\omega}([0,\epsilon];H^{\infty}(\R))\subset C([0,\epsilon];H^{\infty}(\R))$$
    and once again Propostion \ref{prop5} says that $\tilde{u}(0) = u(T^{\ast})$, which means that $\tilde{u}(t) = u(T^{\ast}+t),$ for $t\in[0,\epsilon].$

    Let $s=T^{\ast}+t$ and then $u(s) = \tilde{u}(s-T^{\ast}),$ for $s\in[T^{\ast},T^{\ast}+\epsilon],$
    that is,
$$u \in C^{\omega}([T^{\ast},T^{\ast}+\epsilon];E_{\delta,m}(\R))\subset C^{\omega}([T^{\ast},T^{\ast}+\epsilon];A(\delta)).$$
Based on the definition of $T^{\ast}$, let $T>0$ be such that $T^{\ast}-\epsilon< T< T^{\ast}$ and the solution $u$ then belongs to $C^{\omega}([0,T];A(\sigma(T)))$ for some $\sigma(T)>0$.

Observe now that if $\sigma'\geq \sigma,$ then $A(\sigma')\subset A(\sigma).$ By defining $\tilde{\sigma} = \min\{\delta,\sigma(T)\}$, then
$$u\in C^{\omega}([0,T];A(\tilde{\sigma})),\quad\text{and}\quad u\in C^{\omega}([T^{\ast}-\epsilon,T^{\ast}+\epsilon];A(\tilde{\sigma})),$$
which says that $T^{\ast}$ cannot be the supremum. As a result of the contradiction, $T^{\ast}$ must be infinite and, for every $T>0$, there exists $r(T)>0$ such that $u\in C^{\omega}([0,T];A(r(T)))$.

\item To conclude Theorem \ref{thr1}, we use a result proved by Barostichi, Himonas and Petronilho in \cite{BHP1}, page 752, stating that item (ii) is enough to guarantee that $u \in C^{\omega}([0,\infty)\times \R)$.
\end{enumerate}

\section{Analytic metric}\label{sec4}

In this section we will proceed with the proof of Theorem \ref{thr1a}. Before that, let us briefly understand the required conditions to prove our result. Given any $\epsilon >0$, the strip $\mathcal{S}=(0,\epsilon)\times \R$ will be endowed with the metric \eqref{metric} if $u$ is a generic solution, that is, $u$ is a solution for which the condition $\omega_1\wedge \omega_2\neq 0$ is satisfied on the solutions of \eqref{nonlocal}. For further details, see \cite{DF,FT, SfF} and references therein.

First of all, let us recall a result from \cite{FT} stating that the only nongeneric solutions of \eqref{nonlocal} are:
\begin{align*}
    \phi(t,x) =& \pm \sqrt{a e^{-x} +b},\quad \text{if}\,\, m_1=-2,\\
    \phi(t,x) =& \pm \sqrt{ae^{2x}+b}\quad\text{or}\quad \phi(t,x) = f(t)e^x,\quad \text{if}\,\,m_1=1,
\end{align*}
where $a$ and $b$ are arbitrary constants and $f$ is an arbitrary smooth function. In particular, we observe that for none of the choices do the functions $\phi(t,\cdot)$ belong to $L^1(\R)$ unless identically zero. In this sense, following the discussion in \cite{SfF} that preceeds Theorem 5.1, it is sufficient to prove that: the solution $u$ does not vanish and $u(t,\cdot)\in L^1(\R)$.

Given an initial data $u_0\in G^{1,s}(\R)$, for $s>3/2$, such that $m_0(x)>0$ is an $L^1(\R)$ function, let $u$ be the unique global analytic solution given by Theorem \ref{thr1} restricted to a time slab $\mathcal{S} = (0,\epsilon)\times \R$, where $\epsilon>0$ is arbitrary. Recall that, from Proposition \ref{prop5}, $u(t,\cdot)\in H^{\infty}(\R)$, which means in particular that $u,u_x,u_{xx}$ go to zero as $|x|\to \infty$. Then from \eqref{diffeo} we have $m(t,\cdot)>0$ for any $t>0$, which means that $u(t,\cdot) = g(\cdot)\ast m(t,\cdot)>0$ and $$\Vert m(t)\Vert_{L^1} = \int_{\R} u(t,x)-u_{xx}(t,x)dx = \int_{\R} u(t,x)dx = \Vert u(t)\Vert_{L^1},$$
which means that $u(t,\cdot)\in L^1(\R)$ if and only if $m(t,\cdot)\in L^1(\R)$. But Theorem 3.1 of \cite{SfF} says that the quantity $\Vert m(t)\Vert_{L^1}$ is independent of time, that is, $\Vert m(t)\Vert_{L^1} = \Vert m_0\Vert_{L^1}$  for any $t$ and, therefore, $m(t,\cdot)\in L^1(\R)$.

To conclude the proof of Theorem \ref{thr1a}, it is enough to observe that $u$ is analytic in $\mathcal{S}$, so the components of the metric are analytic.

%\section{Discussion}

%The global well-posedness in Sobolev spaces assumes $s>3/2$. However, because when $\sigma\to 0$ we recover $H^s(\R)=G^{0,s}(\R)$, Theorem \ref{thr2} suggests that $s>3/2$ is not sharpest in the sense that \eqref{nonlocal} might still be (globally) well-posed for $1/2<s\leq 3/2$. The non-uniqueness procedure that usually is proved for these type of equations considers interactions of peakon solutions, which do not exist, at least not in the 'canonical form' $u(t,x) = K e^{-|x-ct|}$, for \eqref{nonlocal}. Therefore, the problem of the sharpest $s$ for well-posedness is indeed a challenging open problem. The immediate implication in this work is that $s>3/2$ might not be sharp and it might be possible to improve the result to $s>1/2$.

\section*{Acknowledgments}

This work was supported by the Royal Society under a Newton International Fellowship (reference number 201625) and by the Conselho Nacional de Desenvolvimento Científico e Tecnológico (process number 308884/2022-1). During the development of this paper, the author was a Lecturer at Universidade Federal do ABC and would also like to thank the institution for the support.

\end{document}